\documentclass[compress,final,3p,times,11pt]{elsarticle}
\usepackage{float}
\usepackage{array}
\usepackage{fancyhdr}
\usepackage{graphicx}
\usepackage{tabularx}
\usepackage{amsmath}
\usepackage{amssymb, amsthm}
\usepackage{mathrsfs}
\usepackage{epstopdf}
\usepackage{lineno,hyperref}
\usepackage{color}
\usepackage{subfigure}
\usepackage{amsfonts}
\usepackage{mathrsfs} % it is the package of using the "\mathfrak{}"
\usepackage{caption}
\usepackage{subfigure}

\newtheorem{definition}{Definition}
\newtheorem{lemma}{Lemma}
\newtheorem{remark}{Remark}
\newtheorem{theorem}{Theorem}

\journal{Axioms}
\begin{document}

%\begin{frontmatter}
\title{Slow manifolds for stochastic Koper models with stable L\'evy noises }

\author[addr1]{Hina Zulfiqar}\ead{zhinazulfiqar@hotmail.com}
\author[addr2]{Shenglan Yuan\corref{cor1}}\ead{shenglan.yuan@math.uni-augsburg.de}\cortext[cor1]{Corresponding author}
\author[addr1]{Muhammad Shoaib Saleem}\ead{shoaib83455@gmail.com}		
		
\address[addr1]{Department of Mathematics, University of Okara, Renala Khurd, Punjab, Pakistan}
\address[addr2]{Institut f\"{u}r Mathematik, Universit\"{a}t Augsburg,
			86135, Augsburg, Germany}

\begin{abstract}
The Koper model is a vector field in which the differential equations describe the
electrochemical oscillations appearing in diffusion processes. This work focuses on the understanding of the slow dynamics of stochastic Koper model perturbed by stable L\'evy noise. We establish the slow manifold for stochastic Koper model with stable L\'evy noise and verify exponential tracking property. We also present a practical example to demonstrate the analytical results with numerical simulations.
\end{abstract}
\maketitle
%\end{frontmatter}
%\linenumbers

\section{Introduction}\label{s:1}
The Koper model \cite{K} is an idealized model of the chemical reaction described in \cite{balabaev2019invariant}. Invariant manifolds are useful in investigating the dynamical behavior of the multiscale systems \cite{bates1998existence,henry2006lecture}. Invariant manifolds for investigating the dynamical behavior of deterministic systems without being influenced of stochastic forces are discussed in \cite{chicone1997center,chow1991smooth,ruelle1982characteristic}, while invariant manifolds for deterministic systems influenced by stochastic forces are constructed in \cite{caraballo2004existence,chow1988invariant,YLZ}. An invariant manifold for a fast-slow stochastic system in which fast mode is indicated by the slow mode tends to slow manifold when the scale parameter goes to zero. Moreover, slow manifold converges to critical manifold as the scale parameter approaches zero.

Slow manifold for stochastic system driven by Brownian motion is demonstrated in \cite{fu2012slow,schmalfuss2008invariant,wang2013slow}, and its numerical simulations are presented in \cite{ren2015approximation,ren2015parameter}. L\'evy motions arise from the models for fluctuations, they have independent, stationary increments and discontinuous paths. For example, L\'evy processes affect the evolution of the state variables in the turbulent flow of fluids \cite{YBD}. Some models about stochastic systems processed by L\'evy noise are explained in \cite{QY,YB,YZD}. The slow manifolds under L\'evy noise are constructed in \cite{yuan2017slow}. The existence of slow manifold for nonlocal fast-slow stochastic evolutionary equations is proved in \cite{zulfiqar2021slow,zulfiqar2019slow}. Invariant manifold of variable stability in the Koper model is established in  \cite{balabaev2019invariant}. It continues to be an active topic on the characterization of stochastic Koper model driven by L\'evy process for both theoretical reasons and applications.

The goal of this article is to construct the three-dimensional stochastic Koper model in Euclidean space $\mathbb{R}^3$ and establish the existence of slow manifold for stochastic Koper model processed by $\alpha$-stable L\'evy noise with $\alpha \in (1,2)$. Namely, we consider the stochastic Koper system in the following version
\begin{equation}\label{K}
\left\{
  \begin{array}{ll}
    \dot{x}=\frac{1}{\epsilon}[ky-x^{3}+3x-\lambda(z)]+\sigma\epsilon^{-\frac{1}{\alpha}}\dot{L}_{t}^{\alpha},  \\
    \dot{y}=x-2y+z,  \\
    \dot{z}=\hat{\epsilon}(y-z),
  \end{array}
\right.
\end{equation}
where $k$ and $\lambda(z)$ are the main bifurcation parameters. Here, $\hat{\epsilon}=1$, but $\epsilon$ represents a small parameter with the property $0<\epsilon\ll1$, and it indicates the ratio of two times scales such that stochastic fast-slow system \eqref{K} has one fast variable $x$, and two slow variables $y$ and $z$. The dot stands for the
differentiation with respect to time $t$. The noise $L_{t}^{\alpha}$ is two-sided symmetric $\alpha$-stable L\'evy process taking real values, with the stability index $\alpha \in (1,2)$ and the intensity $\sigma>0$; see the references \cite{applebaum2009levy,chow1991smooth}.

We start from a random transformation such that a solution of stochastic Koper model \eqref{K} can be expressed as a transformed solution of some random system. The establishment of slow manifold for defined random system is proved by the utilization of Lyapunov-Perron method \cite{chow1988invariant,duan2004smooth}.

The article is organized as follows. In Section \ref{P}, some concepts about random dynamical system, and stochastic differential equation processed by L\'evy motion are discussed. In Section \ref{SA}, the stability of the stochastic Koper system \eqref{K} is proved and a random transformation is defined, which converts stochastic Koper system \eqref{K} into random system. In Section \ref{Rsm}, a short review about random invariant manifold and the existence of exponential tracking slow manifold for random system is provided. In Section \ref{E}, we present numerical
results using an example from electrochemical oscillations to corroborate our analytical
results. Finally, Section \ref{Cf} summarizes our findings as well as directions for future study.

\section{Preliminaries}\label{P}
In this section, some concepts about the random dynamical system are given.
\begin{definition}
Let $(\Omega, \mathcal{F},\mathbb{P})$ be a probability space and $\theta=\{\theta_{t}\}_{t\in\mathbb{R}}$ be a flow on $\Omega$ satisfying the conditions\\
 $\bullet$  $\theta_{0}={\rm Id}_{\Omega};$\\
   $\bullet$ $\theta_{t_{1}}\theta_{t_{2}}=\theta_{t_{1}+t_{2}},$ where $t_{1},t_{2}\in\mathbb{R};$\\and the mapping  $\theta:\mathbb{R}\times\Omega\rightarrow\Omega$ can be defined by  $(t,\omega)\mapsto \theta_{t}\omega$, which is $\mathcal{B}(\mathbb{R})\otimes\mathcal{F}-\mathcal{F}$ measurable. Here, we consider that the probability measure $\mathbb{P}$ is invariant with respect to the flow $\{\theta_{t}\}_{t\in\mathbb{R}}$, i.e., $\theta_{t}\mathbb{P}=\mathbb{P}$ for all $t\in\mathbb{R}$. Then $\Theta=(\Omega, \mathcal{F},\mathbb{P},\theta)$ is said to be a metric dynamical system \cite{yuan2017slow}.
\end{definition}

Throughtout this article, we use scalar L\'evy process. Take $L_{t}^{\alpha}$, $\alpha\in (1,2)$ as a symmetric two-sided $\alpha$-stable L\'evy process with values in $\mathbb{R}$. Consider a canonical sample space for it. Let $\Omega=D(\mathbb{R},\mathbb{R})$ be the space of c\`adl\`ag functions having zero value at $t=0$, i.e.,
\begin{equation*}
D(\mathbb{R},\mathbb{R})=\big\{\omega :\ \mbox{for } \forall~t \in\mathbb{R},\ \lim_{s\uparrow t}\omega(s)=\omega(t-),\  \lim_{s\downarrow t}\omega(s)=\omega(t)\ \mbox{exist}\ \mbox{and} \ \omega(0)=0\big\}.
\end{equation*}
If we use standard usual open-compact metric, then the space $D(\mathbb{R},\mathbb{R})$ may not be separable and complete. But the space $D(\mathbb{R},\mathbb{R})$ of real-valued c\`adl\`ag functions can be extended by introducing another metric ${\rm d}^{0}$ since it can be made into the complete and separable space on unit interval or on $\mathbb{R}$  \cite{chao2018stable}. For functions $\omega_{1}, \omega_{2}\in D(\mathbb{R},\mathbb{R})$, ${\rm d}^{0}(\omega_{1}, \omega_{2})$ is defined by
\begin{eqnarray*}
{\rm d}^{0}(\omega_{1},\omega_{2})=\inf\left\{\varepsilon>0:
     |\omega_{1}(t)-\omega_{2}(\lambda t)|\leq\varepsilon,\ \big|\ln\frac{\arctan(\lambda t)-\arctan(\lambda s)}{\arctan(t)-\arctan(s)}\big|\leq\varepsilon,\right.
     \\
    \left.\mbox{for every}~t, s\in\mathbb{R}\ \mbox{and some}\ \lambda\in\Lambda^{\mathbb{R}}\right\},
\end{eqnarray*}
where
\begin{equation*}
\Lambda^{\mathbb{R}}=\{\lambda :\mathbb{R}\rightarrow\mathbb{R};\ \lambda\ \mbox{is injective increasing},\ \lim\limits_{t\rightarrow-\infty}\lambda(t)=-\infty,\ \lim\limits_{t\rightarrow\infty}\lambda(t)=\infty \}.
\end{equation*}
By Theorem 3.2 in \cite{wei2016weak}, the space $D(\mathbb{R},\mathbb{R})$ equipped with Skorokhod's $\mathcal{J}_{1}$-topology generated by the metric $\rm{d}^{0}$ is a complete and separable space, i.e., Polish space. On this Polish space, we consider a measurable flow $\theta=\{\theta_{t}\}_{t\in\mathbb{R}}$ defined by mapping
 \begin{align*}
 \theta:\mathbb{R}\times D(\mathbb{R},\mathbb{R})\rightarrow D(\mathbb{R},\mathbb{R}),\mbox{ such that}, \theta_{t}\omega(\cdot)=\omega(\cdot+t)-\omega(t),\end{align*}
 where $\omega\in D(\mathbb{R},\mathbb{R})$.

The sample paths of L\'evy process are in $D(\mathbb{R},\mathbb{R})$. Assume that $\mathbb{P}$ is the probability measure on $\mathcal{F}=\mathcal{B}(D(\mathbb{R},\mathbb{R}))$ introduced by the distribution of symmetric two-sided $\alpha$-stable L\'evy process. We consider the restriction of $\mathbb{P}$ on $\mathcal{F}$, but still it is indicated by $\mathbb{P}$. Observe that $\mathbb{P}$ is ergodic with respect to the shift $\{\theta_{t}\}_{t\in\mathbb{R}}$. Thus $(D(\mathbb{R},\mathbb{R}), \mathcal{F}, \mathbb{P}, \{\theta_{t}\}_{t\in\mathbb{R}})$ is a metric dynamical system. It is worth pointing out that we can take a subset $\Omega_{1}=D^{0}(\mathbb{R},\mathbb{R})\subset \Omega=D(\mathbb{R},\mathbb{R})$ with $\mathbb{P}$-measure one instead of $D(\mathbb{R},\mathbb{R})$. Here, $D^{0}(\mathbb{R},\mathbb{R})$ is $\{\theta_{t}\}_{t\in\mathbb{R}}$-invariant, which means that $\theta_{t}\Omega_{1}=\Omega_{1}$ for $t\in \mathbb{R}$.

\begin{definition}
 A cocycle $\phi$ satisfies that \begin{align*}
&\phi(0,\omega,u)=u;\\
&\phi(t_{1}+t_{2},\omega,u)=\phi(t_{2},\theta_{t_{1}}\omega,\phi(t_{1},\omega,u)).
\end{align*}
It is $\mathcal{B}(\mathbb{R})\otimes\mathcal{F}\otimes\mathcal{B}(\mathbb{R}^{3})-\mathcal{F}$ measurable and defined by map:
$$\phi:\mathbb{R}\times\Omega\times\mathbb{R}^{3}\rightarrow\mathbb{R}^{3},$$ for $u\in\mathbb{R}^{3}$, $\omega\in\Omega$ and $t_{1},t_{2}\in\mathbb{R}$. Metric dynamical system $(\Omega,\mathcal{F},\mathbb{P},\theta)$ with the cocycle $\phi$ generates a random dynamical system \cite{arnold2013random}.
\end{definition}

The above cocycle property indicates that random dynamical system $\phi$ arrives at the same destinaton whether we consider the position $\phi(t_{1}+t_{2},\omega,u)$ of
the path starting in $u$ at time $t_1+t_2$ or the position $\phi(t_{2},\theta_{t_{1}}\omega,\phi(t_{1},\omega,u))$ of the path with random initial state $\phi(t_{1},\omega,u)$ at time $t_2$. It is important to note that the path moves from $u$ to $\phi(t_{1},\omega,u)$, the underlying $\omega$ potentially may change as well over
time $t_{1}$. Instead of $\omega$ we need to use $\theta_{t_{1}}\omega$ for the new movement with the starting point $\phi(t_{1},\omega,u)$, where $\theta_{t_{1}}$ indicates the new development of the underlying probability space over time $t_{1}$.

If $u\mapsto\phi(t,\omega,u)$ is continuous or differentiable for $t\in\mathbb{R}$ and $\omega\in\Omega$, then random dynamical system $(\Omega,\mathcal{F},\mathbb{P},\theta, \phi)$ is also continuous or differentiable. The family of nonempty closed sets $\mathcal{M}=\{\mathcal{M}(\omega)\subset\mathbb{R}^3:\omega\in \Omega\}$ is said to be a random set if for all $u'\in\mathbb{R}^3$ the map:$$\omega\mapsto\mathop {\inf }\limits_{u \in \mathcal{M}(\omega)}|u-u'|,$$ is a random variable.

\begin{definition}
If random variable $u(\omega)$ taking values in $\mathbb{R}^3$, satisfies $$\phi(t,\omega,u(\omega))=u(\theta_{t}\omega),\indent \indent a.s.$$for all $t\in\mathbb{R}$. Then, random variable $u(\omega)$ is known as stationary orbit or random fixed point \cite{duan2015introduction}.
\end{definition}

\begin{definition}
 A random set $\mathcal{M}=\{\mathcal{M}(\omega)\subset\mathbb{R}^3:\omega\in \Omega\}$ is called random positively invariant set \cite{fu2012slow} for random dynamical system $\phi$, if$$\phi(t,\omega,\mathcal{M}(\omega))\subset \mathcal{M}(\theta_{t}\omega),$$ for all $\omega\in\Omega$ and $t\geq0$.
\end{definition}

\begin{definition}
 Introduce a map $$l:{\mathbb{R}^{2}}\times\Omega\rightarrow {\mathbb{R}},$$ such that $v\mapsto l(v,\omega)$ is Lipschitz continuous for all $\omega\in\Omega$. Consider $$\mathcal{M}(\omega)=\{(l(v,\omega),v): v\in{\mathbb{R}^{2}}\},$$
such that random positively invariant set $\mathcal{M}=\{\mathcal{M}(\omega)\subset\mathbb{R}^3:\omega\in \Omega\}$ can be expressed as a graph of Lipschitz continuous map $l$, then $\mathcal{M}$ is known as Lipschitz continuous invariant manifold \cite{yuan2017slow}.
\end{definition}

Moreover, $\mathcal{M}$ posses the exponential tracking property, if there exists an $u'\in \mathcal{M}(\omega)$ for all $u\in \mathbb{R}^3$ satisfying $$|\phi(t,\omega,u)-\phi(t,\omega,u')|\leq c_{1}(u,u',\omega)e^{c_{2}t}|u-u'|,$$
for all $\omega\in \Omega$. Here, $c_{1}$ is positive random variable, and $c_{2}$ is negative constant.

 \section{Stability Analysis}\label{SA}
Stochastic Koper system \eqref{K} is consist of one fast mode and two slow modes. The state space for the fast mode is $\mathbb{R}$, and the state space for the slow modes is $\mathbb{R}^{2}$. For the construction of slow manifold, we assume the following hypotheses for the Koper system  \eqref{K}.

\textbf{(H1)} (Lipschitz continuity) With regard to nonlinear parts of \eqref{K}, there are positive constants $L_{1}, L_{2}, L_{3}>0$ such that for all $(x_{i},y_{i},z_{i})^{T}$ in $\mathbb{R}^3$ and for all $(x_{j},y_{j},z_{j})^{T}$ in $\mathbb{R}^3$,
\begin{align*}
&|g_{1}(x_{i},y_{i},z_{i})-g_{1}(x_{j},y_{j},z_{j})|+|g_{2}(x_{i},y_{i},z_{i})-g_{2}(x_{j},y_{j},z_{j})|+|g_{3}(x_{i},y_{i},z_{i})-g_{3}(x_{j},y_{j},z_{j})| \\
&\leq L_{1}(|x_{i}-x_{j}|+|y_{i}-y_{j}|+|z_{i}-z_{j}|),
\end{align*}
where $T$ is the transpose of the vector, and $g_{m}: \mathbb{R}^{3}\rightarrow \mathbb{R}, m=1,2,3$ are defined by
 $g_{1}(x,y,z)=ky-x^{3}+3x-\lambda(z), g_{2}(x,y,z)=x-2y+z,$ and $g_{3}(x,y,z)=y-z$.

\textbf{(H2)} (Growth) For all $(x,y,z)\in \mathbb{R}^{3}$, there exists a positive constant $L$ such that
$$ |g_{1}(x,y,z)|^{2}+|g_{2}(x,y,z)|^{2}+|g_{3}(x,y,z)|^{2}\leq L( 1+|x|^{2}+|y|^{2}+|z|^{2}).$$

\textbf{(H3)} (Monotonocity) For all $x_{1},x_{2}\in \mathbb{R}$, there exists a positive constant $L$ such that
$$ (x_{2}-x_{1})(g_{1}(x_{2})-g_{1}(x_{1}))\leq -L(|x_{2}-x_{1}|^{2}).$$\\

Now let $\Theta_{1}=(\Omega_{1},\mathcal{F}_{1},\mathbb{P}_{1},\theta_{t}^{1})$, $\Theta_{2}=(\Omega_{2},\mathcal{F}_{2},\mathbb{P}_{2},\theta_{t}^{2})$ and $\Theta_{3}=(\Omega_{3},\mathcal{F}_{3},\mathbb{P}_{3},\theta_{t}^{3})$  be three independent driving (metric) dynamical systems as mentioned in Section 2. Define
\begin{align*}\Theta=\Theta_{1}\times\Theta_{2}\times\Theta_{3}=(\Omega_{1}\times\Omega_{2}\times\Omega_{3},\mathcal{F}_{1}\otimes\mathcal{F}_{2}\otimes\mathcal{F}_{3},\mathbb{P}_{1}\times\mathbb{P}_{2}\times\mathbb{P}_{3},(\theta_{t}^{1},\theta_{t}^{2},\theta_{t}^{3})^{T}),\end{align*}
and $$\theta_{t}\omega:=(\theta_{t}^{1}\omega_{1},\theta_{t}^{2}\omega_{2}, \theta_{t}^{3}\omega_{3})^{T}, \mbox{ for } \omega:=(\omega_{1},\omega_{2},\omega_{3})^{T}\in \Omega:=\Omega_{1}\times\Omega_{2}\times\Omega_{3}.$$
Let $L_{t}^{\alpha}, \alpha\in(1,2)$ be a two-sided symmetric $\alpha$-stable L\'evy process in $\mathbb{R}$ with a generating triplet $(a,\mathcal{Q},v)$. We will prove the existence and uniqueness of solution for the stochastic Koper system \eqref{K}.

\begin{lemma} \cite{protter2004stochastic} Under Lipschitz condition ${\rm \textbf{(H1)}}$, the equation
\begin{equation}\label{eq2}
d\delta(t)=g_{1}(\delta(t))dt+\sigma dL_{t}^{\alpha},\indent \delta(0)=\delta_{0},
\end{equation}
has the unique c\`adl\`ag solution
\begin{align*}
\delta(t)=\delta_0+\int_0^{t}g_1(\delta(s))ds+\sigma L_{t}^{\alpha}.
\end{align*}
According to the L\'evy-It\^o decomposition, $L_t^{\alpha}$ can be expressed as
\begin{equation*}
 L_{t}^{\alpha}=\int_{|w|<1}w\tilde{N}(t,dw)+\int_{|w|\geq 1}wN(t,dw).
\end{equation*}
It follows that
\begin{equation*}
\delta(t)=\delta_0+\int_0^{t}g_1(\delta(s))ds+\sigma\int_{|w|<1}w\tilde{N}(t,dw)+\sigma\int_{|w|\geq 1}wN(t,dw).
\end{equation*}
\end{lemma}
%\begin{lemma} \cite{yuan2017slow} Consider the stochastic evolutionary equation
%\begin{equation}\label{eq3}
%d\eta(t)=\frac{1}{\epsilon}g_{1}(\eta(t))dt+\epsilon^{-\frac{1}{\alpha}}dL_{t}^{\alpha}.
%\end{equation}
%Assume that the hypothesis ${\rm \textbf{(H1)}}$ holds. Then equation \eqref{eq3} has unique solution, say, $\eta^{\epsilon}(\theta_{t}\omega)$.\end{lemma}
\begin{remark} (\cite[p.191]{duan2015introduction}) $L_{ct}^{\alpha}$ and $c^{\frac{1}{\alpha}}L_{t}^{\alpha}$ have the same distribution for every $c>0$, i.e., $L_{ct}^{\alpha}\overset{d}{=} c^{\frac{1}{\alpha}}L_{t}^{\alpha}.$
\end{remark}

\begin{lemma} Under the assumptions {\rm \textbf{(H1)}-\textbf{(H3)}}, stochastic Koper system \eqref{K} has a unique solution.\end{lemma}

\begin{proof}
Rewrite the stochastic Koper system \eqref{K} into the form
\begin{align*} \left( {\begin{array}{*{20}{c}}
{\dot{x}}\\
\dot{y}\\
\dot{z}\\
\end{array}} \right)=
\left( {\begin{array}{*{20}{c}}
{\frac{1}{\epsilon}g_{1}(x,y,z)}\\
g_{2}(x,y,z)\\
g_{3}(x,y,z)\\
\end{array}} \right)
+\left( {\begin{array}{*{20}{c}}
{\sigma\epsilon^{-\frac{1}{\alpha}}\dot{L}_{t}^{\alpha}}\\
0\\0\\
\end{array}} \right)
\end{align*}%\nocite{*}
From \cite{xiaoyu2022central}, under the assumptions {\rm \textbf{(H1)}-\textbf{(H3)}}, it implies from \cite[Theorem III.2.3.2]{jacod2013limit} that there exists a unique solution of the equation \eqref{eq2}.
Moreover, there also exists one exponentially mixing invariant measure with respect to transition semigroup of $x(t)$.  Then by \cite[chapter 6]{applebaum2009levy}, the assumptions ${\rm \textbf{(H1)}-\textbf{(H3)}}$ indicate that there exists a unique mild solution $(x(t), y(t), z(t))^{T}$ in $\mathbb{R}^{3}$ for the stochastic Koper system \eqref{K}.
\end{proof}
Define a random transformation
\begin{align*}\left( {\begin{array}{*{20}{c}}
{X}\\
Y\\
Z\\
\end{array}} \right):=\mu(\theta_{t}\omega,x,y,z):=\left( {\begin{array}{*{20}{c}}
x-\sigma\eta^{\epsilon}(\theta_{t}\omega)\\
y\\
z\\
\end{array}} \right).
\end{align*}
where $\eta^{\epsilon}(\theta_{t}\omega):=\epsilon^{-\frac{1}{\alpha}}L_{t}^{\alpha}(\omega)$. Then $(X(t),Y(t),Z(t))=\mu(\theta_{t}\omega,x,y,z)$ satisfies the random system
\begin{equation}\label{eq4}
\left\{
  \begin{array}{ll}
dX=\frac{1}{\epsilon}g_{1}(X+\sigma\eta^{\epsilon}(\theta_{t}\omega),Y,Z)dt,  \\
dY=g_{2}(X+\sigma\eta^{\epsilon}(\theta_{t}\omega),Y, Z)dt, \\
dZ=g_{3}(X+\sigma\eta^{\epsilon}(\theta_{t}\omega),Y,Z)dt.
\end{array}
\right.
\end{equation}
The term $\sigma\eta^{\epsilon}(\theta_{t}\omega)$ do not change the Lipschitz constants of $g_{1}$, $g_{2}$ and $g_{3}$. So $g_{1}$, $g_{2}$ and $g_{3}$ in random dynamical system \eqref{eq4} and in stochastic dynamical system \eqref{K} have the same Lipschitz constants.
The random system \eqref{eq4} can be solved for any $\omega\in \Omega$. So, for any initial value $(X(0),Y(0),Z(0))^{T}=(X_{0},Y_{0},Z_{0})^{T}$, the solution operator
\begin{align*}
(t,\omega,(X_{0},Y_{0},Z_{0})^{T})&\mapsto\Phi(t,\omega,(X_{0},Y_{0},Z_{0})^{T})\\
&=\big(X(t,\omega,(X_{0},Y_{0},Z_{0})^{T}),Y(t,\omega,(X_{0},Y_{0},Z_{0})^{T}),Z(t,\omega,(X_{0},Y_{0},Z_{0})^{T})\big)^{T},
\end{align*}
represents the random dynamical system for \eqref{eq4}. Furthermore,
\begin{align*}
\phi(t,\omega,(X_{0},Y_{0},Z_{0})^{T})=\Phi(t,\omega,(X_{0},Y_{0},Z_{0})^{T})+(\sigma\eta^{\epsilon}(\theta_{t}\omega),0,0)^{T},
\end{align*}
characterizes the random dynamical system generated by stochastic Koper system \eqref{K}.

\section{Random slow manifolds}\label{Rsm}
Introduce the Banach spaces of functions for investigating the random system \eqref{eq4}. For any $\beta\in\mathbb{R}$,
\begin{align*}
C_{\beta}^{\mathbb{R},-}&:=\Big\{\Phi:(-\infty,0]\rightarrow \mathbb{R}\mbox{ is continuous and } \mathop {\mbox{sup} }\limits_{t\in(-\infty,0]}|e^{-\beta t}\Phi(t)|<\infty\Big\},\\
C_{\beta}^{\mathbb{R},+}&:=\Big\{\Phi:[0,\infty)\rightarrow \mathbb{R}\mbox{ is continuous and } \mathop {\mbox{sup} }\limits_{t\in[0,\infty)}|e^{-\beta t}\Phi(t)|<\infty\Big\},
\end{align*}
with norms
\begin{equation*}
||\Phi||_{C_{\beta}^{\mathbb{R},-}}:=\mathop {\mbox{sup} }\limits_{t\in(-\infty,0]}|e^{-\beta t}\Phi(t)|,\quad \mbox{ and } \quad ||\Phi||_{C_{\beta}^{\mathbb{R},+}}:=\mathop {\mbox{sup} }\limits_{t\in[0,\infty)}|e^{-\beta t}\Phi(t)|.
\end{equation*}
Let $C_{\beta}^{\mathbb{R}^3,\pm}$ be the product of spaces $C_{\beta}^{\mathbb{R}^3,\pm}:=C_{\beta}^{\mathbb{R},\pm}\times C_{\beta}^{\mathbb{R},\pm}\times C_{\beta}^{\mathbb{R},\pm}$,
with norm
\begin{align*}
||U||_{C_{\beta}^{\mathbb{R}^3,\pm}}=||X||_{C_{\beta}^{\mathbb{R},\pm}}+||Y||_{C_{\beta}^{\mathbb{R},\pm}}+||Z||_{C_{\beta}^{\mathbb{R},\pm}},\indent U=(X,Y,Z)^{T}\in C_{\beta}^{\mathbb{R}^3,\pm}.
\end{align*}

For $U(\cdot,\omega)=\big(X(\cdot,\omega),Y(\cdot,\omega),Z(\cdot,\omega)\big)^{T}\in C_{\beta}^{\mathbb{R}^3,-}$, it is the solution of \eqref{eq4} with initial value $U_{0}=(X_{0},Y_{0},Z_{0})^{T}$ iff $U(t,\omega)$ satisfies
\begin{equation}\label{sde}
\left( {\begin{array}{*{20}{c}}
{X(t)}\\
Y(t)\\
Z(t)\\
\end{array}} \right)=
\left( {\begin{array}{*{20}{c}}
{\frac{1}{\epsilon}\int_{-\infty}^{t}g_{1}\big(X(s)+\sigma\eta^{\epsilon}(\theta_{s}\omega),Y(s),Z(s)\big)ds}\\
Y_{0}+\int_{0}^{t}g_{2}\big(X(s)+\sigma\eta^{\epsilon}(\theta_{s}\omega),Y(s),Z(s)\big)ds\\
Z_{0}+\int_{0}^{t}g_{3}\big(X(s)+\sigma\eta^{\epsilon}(\theta_{s}\omega),Y(s),Z(s)\big)ds\\
\end{array}} \right).
\end{equation}%\nocite{*}
\begin{lemma}\label{Uu}
Suppose that $$U(t,\omega,U_{0})=\Big(X\big(t,\omega,(X_{0},Y_{0},Z_{0})^{T}\big) ,Y\big(t,\omega,(X_{0},Y_{0},Z_{0})^{T}\big),Z\big(t,\omega,(X_{0},Y_{0},,Z_{0})^{T}\big)\Big)^{T}$$ is the solution of  \eqref{sde} with $t\leq 0$.
Then $U(t,\omega,U_{0})$ is the unique solution in $C_{\beta}^{\mathbb{R}^3,-}$, where $U_{0}=(X_{0},Y_{0},Z_{0})^{T}$ is the initial value.
\end{lemma}

\begin{proof} By using Banach fixed point theorem, we prove that $$U(t,\omega,U_{0})=\Big(X\big(t,\omega,(X_{0},Y_{0},Z_{0})^{T}\big) ,Y\big(t,\omega,(X_{0},Y_{0},Z_{0})^{T}\big),Z\big(t,\omega,(X_{0},Y_{0},Z_{0})^{T}\big)\Big)^{T}$$ is the unique solution of (\ref{sde}). For the prove of it, define three operators for $t\leq0$:
$$\mathfrak{J}_{1}(U)[t]=\frac{1}{\epsilon}\int_{-\infty}^{t}g_{1}\big(X(s)+\sigma\eta^{\epsilon}(\theta_{s}\omega),Y(s),Z(s)\big)ds,$$
$$\mathfrak{J}_{2}(U)[t]=Y_{0}+\int_{0}^{t}g_{2}\big(X(s)+\sigma\eta^{\epsilon}(\theta_{s}\omega),Y(s),Z(s)\big)ds,$$
$$\mathfrak{J}_{3}(U)[t]=Z_{0}+\int_{0}^{t}g_{3}\big(X(s)+\sigma\eta^{\epsilon}(\theta_{s}\omega),Y(s),Z(s)\big)ds.$$
Then Lyapunov-Perron transform is defined to be
$$\mathfrak{J}(U)=\left( {\begin{array}{*{20}{c}}
{\mathfrak{J}_{1}(U)}\\
\mathfrak{J}_{2}(U)\\
\mathfrak{J}_{3}(U)\\
\end{array}} \right)
=\big(\mathfrak{J}_{1}(U),\mathfrak{J}_{2}(U),\mathfrak{J}_{3}(U)\big)^{T}.$$
Now, it is needed to prove that $\mathfrak{J}$ maps $C_{\beta}^{\mathbb{R}^3,-}$ into itself.
Take $U=(X,Y,Z)^{T}\in C_{\beta}^{\mathbb{R}^3,-}$ satisfying:
\begin{align*}||\mathfrak{J}_{1}(U)[t]||_{C_{\beta}^{\mathbb{R},-}}&=||\frac{1}{\epsilon}\int_{-\infty}^{t}g_{1}\big(X(s)+\sigma\eta^{\epsilon}(\theta_{s}\omega),Y(s),Z(s)\big)ds||_{C_{\beta}^{\mathbb{R},-}}\\
&\leq\frac{1}{\epsilon}\mathop { \mbox{sup} }\limits_{t\in (-\infty,0]}\Big\{e^{-\beta t}\int_{-\infty}^{t}\big|g_{1}\big(X(s)+\sigma\eta^{\epsilon}(\theta_{s}\omega),Y(s),Z(s)\big)\big|ds\Big\}\\
&\leq\frac{K}{\epsilon}\mathop {\mbox{sup} }\limits_{t\in (-\infty,0]}\Big\{e^{-\beta t}\int_{-\infty}^{t}\big(|X(s)|+|Y(s)|+|Z(s)|\big)ds\Big\}+ \mathcal{C}_{1}\\ &\leq\frac{K}{\epsilon}\mathop {\mbox{sup} }\limits_{t\in(-\infty,0]}\Big\{\int_{-\infty}^{t}e^{-\beta(t-s)}ds\Big\}||U||_{C_{\beta}^{\mathbb{R}^3,-}}+ \mathcal{C}_{1}\\&=\frac{K}{-\epsilon\beta}||U||_{C_{\beta}^{\mathbb{R}^3,-}}+ \mathcal{C}_{1}.\end{align*}
Similarly, we have
\begin{align*}||\mathfrak{J}_{2}(U)[t]||_{C_{\beta}^{\mathbb{R},-}}&=||Y_{0}+\int_{0}^{t}g_{2}\big(X(s)+\sigma\eta^{\epsilon}(\theta_{s}\omega),Y(s),Z(s)\big)ds||_{C_{\beta}^{\mathbb{R},-}}\\
&\leq\mathop { \mbox{sup} }\limits_{t\in (-\infty,0]}\Big\{e^{-\beta t}\int_{0}^{t}\big|g_{2}\big(X(s)+\sigma\eta^{\epsilon}(\theta_{s}\omega),Y(s),Z(s)\big)\big|ds\Big\}+||Y_{0}||_{C_{\beta}^{\mathbb{R},-}}\\
&\leq\mathop {\mbox{sup} }\limits_{t\in (-\infty,0]}\Big\{e^{-\beta t}\int_{0}^{t}\big(|X(s)|+|Y(s)|+|Z(s)|\big)ds\Big\}+ \mathcal{C}_{2}\\ &\leq K\mathop {\mbox{sup} }\limits_{t\in(-\infty,0]}\Big\{\int_{0}^{t}e^{-\beta(t-s)}ds\Big\}||U||_{C_{\beta}^{\mathbb{R}^3,-}}+ \mathcal{C}_{2}\\&=\frac{K}{\beta}||U||_{C_{\beta}^{\mathbb{R}^3,-}}+ \mathcal{C}_{2}.\end{align*}
Furthermore,
\begin{align*}||\mathfrak{J}_{3}(U)[t]||_{C_{\beta}^{\mathbb{R},-}}&=||Z_{0}+\int_{0}^{t}g_{3}\big(X(s)+\sigma\eta^{\epsilon}(\theta_{s}\omega),Y(s),Z(s)\big)ds||_{C_{\beta}^{\mathbb{R},-}}\\
&\leq\mathop { \mbox{sup} }\limits_{t\in (-\infty,0]}\Big\{e^{-\beta t}\int_{0}^{t}\big|g_{3}\big(X(s)+\sigma\eta^{\epsilon}(\theta_{s}\omega),Y(s),Z(s)\big)\big|ds\Big\}+||Z_{0}||_{C_{\beta}^{\mathbb{R},-}}\\
&\leq\mathop {\mbox{sup} }\limits_{t\in (-\infty,0]}\Big\{e^{-\beta t}\int_{0}^{t}\big(|X(s)|+|Y(s)|+|Z(s)|\big)ds\Big\}+ \mathcal{C}_{3}\\ &\leq K\mathop {\mbox{sup} }\limits_{t\in(-\infty,0]}\Big\{\int_{0}^{t}e^{-\beta(t-s)}ds\Big\}||U||_{C_{\beta}^{\mathbb{R}^3,-}}+ \mathcal{C}_{3}\\&=\frac{K}{\beta}||U||_{C_{\beta}^{\mathbb{R}^3,-}}+ \mathcal{C}_{3}.\end{align*}
By using Lyapunov-Perron transform, the estimate of $\mathfrak{J}$ in combined form is
\begin{equation*}
||\mathfrak{J}(U)||_{C_{\beta}^{\mathbb{R}^3,-}}\leq\varrho(\beta,K,\epsilon)||U||_{C_{\beta}^{\mathbb{R}^3,-}}+\mathcal{C},
\end{equation*}
where
\begin{equation*}
\varrho(\beta,K,\epsilon)= \frac{-K}{\epsilon\beta}+\frac{K}{\beta}+\frac{K}{\beta}.
\end{equation*}
Hence $\mathfrak{J}(U)$ is in $C_{\beta}^{\mathbb{R}^3,-}$ for all $U\in C_{\beta}^{\mathbb{R}^3,-}$, which means that $\mathfrak{J}$ maps $C_{\beta}^{\mathbb{R}^3,-}$ into itself.

Now we should show that the map $\mathfrak{J}$ is contractive. For this, take $U=(X,Y,Z)^{T},\tilde{U}=(\tilde{X},\tilde{Y},\tilde{Z})^{T}\in C_{\beta}^{\mathbb{R}^3,-}$,
\begin{align*}&||\mathfrak{J}_{1}(U)-\mathfrak{J}_{1}(\tilde{U})||_{C_{\beta}^{\mathbb{R},-}}\\
&\leq\frac{1}{\epsilon}\mathop {\mbox{sup} }\limits_{t\in (-\infty,0]}\Big\{e^{-\beta t}\int_{-\infty}^{t}\Big|g_1\big(X(s)+\sigma\eta^{\epsilon}(\theta_{s}\omega),Y(s),Z(s)\big)-g_1\big(\tilde{X}(s)+\sigma\eta^{\epsilon}(\theta_{s}\omega),\tilde{Y}(s),\tilde{Z}(s)\big)\Big|ds\Big\}\\
&\leq\frac{K}{\epsilon}\mathop {\mbox{sup} }\limits_{t\in(-\infty,0]}\Big\{e^{-\beta t }\int_{-\infty}^{t}\big(|X(s)-\tilde{X}(s)|+|Y(s)-\tilde{Y}(s)|+|Z(s)-\tilde{Z}(s)|\big)ds\Big\}\\
&\leq\frac{K}{\epsilon}\mathop {\mbox{sup} }\limits_{t\in(-\infty, 0]}\Big\{\int_{-\infty}^{t}e^{-\beta(t-s)}ds\Big\}||U-\tilde{U}||_{C_{\beta}^{\mathbb{R}^3,-}}\\
&=\frac{K}{-\epsilon\beta}||U-\tilde{U}||_{C_{\beta}^{\mathbb{R}^3,-}}.\end{align*}
Using the same way,
\begin{align*}
||\mathfrak{J}_{2}(U)-\mathfrak{J}_{2}(\tilde{U})||_{C_{\beta}^{\mathbb{R},-}}&\leq K \mathop {\mbox{sup} }\limits_{t\in(-\infty, 0]}\Big\{\int_{t}^{0}e^{-\beta(t-s)}ds\Big\}||U-\tilde{U}||_{C_{\beta}^{\mathbb{R}^3,-}}\\
&\leq\frac{ K}{\beta}||U-\tilde{U}||_{C_{\beta}^{\mathbb{R}^3,-}}.
\end{align*}
Moreover,
\begin{align*}
||\mathfrak{J}_{3}(U)-\mathfrak{J}_{3}(\tilde{U})||_{C_{\beta}^{\mathbb{R},-}}&\leq K \mathop {\mbox{sup} }\limits_{t\in(-\infty, 0]}\Big\{\int_{t}^{0}e^{-\beta(t-s)}ds\Big\}||U-\tilde{U}||_{C_{\beta}^{\mathbb{R}^{3},-}}\\
&\leq \frac{ K}{\beta}||U-\tilde{U}||_{C_{\beta}^{\mathbb{R}^{3},-}}.
\end{align*}
Combing the three together,
\begin{equation*}
||\mathfrak{J}(U)-\mathfrak{J}(\tilde{U})||_{C_{\beta}^{\mathbb{R}^{3},-}}\leq\varrho(\beta,K,\epsilon)||U-\tilde{U}||_{C_{\beta}^{\mathbb{R}^{3},-}},
\end{equation*}
where\begin{align*}&\varrho(\beta,K,\epsilon)=\frac{K}{-\epsilon\beta}+\frac{K}{\beta}+\frac{K}{\beta}.\end{align*}
By setting $\beta=-\frac{\gamma}{\epsilon}$,
\begin{align*}\varrho(\beta,K,\epsilon)\rightarrow\frac{K}{\gamma} \mbox{ for } \epsilon\rightarrow0.\end{align*}
So, there exists a sufficiently small $\epsilon_{0}\rightarrow 0$ with property
\begin{align*}0<\varrho(\beta,K,\epsilon)< 1,\mbox{ for  }\epsilon \in  (0,\epsilon_{0}).\end{align*} Hence, the map $\mathfrak{J}$ in $C_{-\frac{\gamma}{\epsilon}}^{\mathbb{R}^{3},-}$ is contractive. By Banach fixed point theorem,  every contractive mapping in Banach space has a unique fixed point. Thus, \eqref{sde} has the unique solution \begin{align*}
U(t,\omega,U_{0})=\big(X(t,\omega,(X_{0},Y_{0},Z_{0})^{T}),Y(t,\omega,(X_{0},Y_{0},Z_{0})^{T}),Z(t,\omega,(X_{0},Y_{0},Z_{0})^{T})\big)^{T} \mbox{ in } C_{-\frac{\gamma}{\epsilon}}^{\mathbb{R}^{3},-}.
\end{align*}
\end{proof}
From Lemma \ref{Uu} we obtain the following remark.\\
\begin{remark} For any $(X_{0},Y_{0}, Z_{0})^{T}$, $(X'_{0},Y'_{0},Z'_{0})^{T}\in \mathbb{R}^3$, and for all $\omega \in \Omega, Y_{0}, Y'_{0},Z_{0}, Z'_{0} \in \mathbb{R},$ there is an $\epsilon_{0}>0$ such that
\begin{align}\label{eq9}
||U\big(t,\omega,(X_{0},Y_{0},Z_{0})^{T}\big)-U\big(t,\omega,(X'_{0},Y'_{0},Z'_{0})^{T}\big)||_{C_{-\frac{\gamma}{\epsilon}}^{\mathbb{R}^{3},-}}\leq\frac{1}{1-\varrho(\beta,K,\epsilon)}\big(|Y_{0}-Y_{0}'|+|Z_{0}-Z_{0}'|\big).\end{align}\end{remark}
\begin{proof} Instead of writing $U\big(t,\omega,(X_{0},Y_{0},Z_{0})^{T}\big)$ and $U\big(t,\omega,(X'_{0},Y'_{0},Z'_{0})^{T}\big)$, we write $U(t,\omega,Y_{0},Z_{0})$ and $U(t,\omega,Y'_{0},Z'_{0})$. For all $\omega \in \Omega$ and $Y_{0}, Y'_{0}, Z_{0}, Z'_{0}\in\mathbb{R}$, we determine an upper bound
 \begin{align*}
&||U(t,\omega,Y_{0},Z_{0})-U(t,\omega,Y'_{0},Z'_{0})||_{C_{-\frac{\gamma}{\epsilon}}^{\mathbb{R}^{3},-}}\\ &=||X(t,\omega,Y_{0},Z_{0})-X(t,\omega,Y'_{0},Z'_{0})||_{C_{-\frac{\gamma}{\epsilon}}^{\mathbb{R},-}}+||Y(t,\omega,Y_{0},Z_{0})-Y(t,\omega,Y'_{0},Z'_{0})||_{C_{-\frac{\gamma}{\epsilon}}^{\mathbb{R},-}}+||Z(t,\omega,Y_{0},Z_{0})-Z(t,\omega,Y'_{0},Z'_{0})||_{C_{-\frac{\gamma}{\epsilon}}^{\mathbb{R},-}}\\
&\leq\big(\frac{K}{-\epsilon\beta}+\frac{ 2K}{\beta}\big)\times||U(t,\omega,Y_{0},Z_{0})-U(t,\omega,Y'_{0},Z'_{0})||_{C_{-\frac{\gamma}{\epsilon}}^{\mathbb{R}^{3},-}} +|Y_{0}-Y'_{0}|+|Z_{0}-Z'_{0}|\\
&=\varrho(\beta,K,\epsilon)||U(t,\omega,Y_{0},Z_{0})-U(t,\omega,Y'_{0},Z'_{0})||_{C_{-\frac{\gamma}{\epsilon}}^{\mathbb{R}^{3},-}}+|Y_{0}-Y'_{0}|+|Z_{0}-Z'_{0}|.
\end{align*}
Therefore \eqref{eq9} is valid.
\end{proof}

%Define
%$$\mathcal{M}^{\epsilon}(\omega)\triangleq\{U_{0}\in \mathbb{R}^3: U(t,\omega,U_{0})\in C_{\beta}^{\mathbb{R}^3,-}\},\mbox{ with %}\beta=-\frac{\gamma}{\epsilon}.$$
Next, with the help of Lyapunov-Perron method we will construct the slow manifold as a random graph.
\begin{theorem}
Assume that the hypotheses {\rm \textbf{(H1)}-\textbf{(H3)}} hold. Then for sufficiently small $\epsilon>0$, random system \eqref{eq4} posses a Lipschitz random slow manifold:$$\mathcal{M}^{\epsilon}(\omega)=\big\{\big(l^{\epsilon}(\omega,Y_{0},Z_{0}),Y_{0},Z_{0}\big)^{T}:Y_{0},Z_{0}\in \mathbb{R}\big\},$$where $$l^{\epsilon}(\cdot,\cdot):\Omega\times \mathbb{R}^{2}\rightarrow\mathbb{R},$$ is a Lipschitz graph map with Lipschitz constant$$\text{\rm Lip} l^{\epsilon}(\omega,\cdot)\leq\frac{K}{\gamma-K(1-2\epsilon)}.$$
\end{theorem}
\begin{proof} For any $Y_{0},Z_{0}\in \mathbb{R}$ we define Lyapunov-Perron map  $l^{\epsilon}$ by
\begin{equation*}
l^{\epsilon}(\omega, Y_{0},Z_{0})=\frac{1}{\epsilon}\int_{-\infty}^{0}g_{1}\big(X(s,\omega,Y_{0},Z_{0})+\sigma\eta^{\epsilon}(\theta_{s}\omega),Y(s,\omega,Y_{0},Z_{0}),Z(s,\omega,Y_{0},Z_{0})\big)ds.
\end{equation*}
Then by \eqref{eq9}, we obtain $$\big|l^{\epsilon}(\omega,Y_{0},Z_{0})-l^{\epsilon}(\omega,Y'_{0},Z'_{0})\big|\leq\frac{K}{-\epsilon\beta}\frac{1}{[1-\varrho(\beta,K,\epsilon)]}\big(|Y_{0}-Y'_{0}|+|Z_{0}-Z'_{0}|\big),$$
for all $Y_{0},Y'_{0},Z_{0},Z'_{0} \in \mathbb{R}$ and $\omega \in \Omega$. So$$\big|l^{\epsilon}(\omega,Y_{0},Z_{0})-l^{\epsilon}(\omega,Y'_{0},Z'_{0})\big|\leq\frac{K}{\gamma}\frac{1}{[1-\varrho(\beta,,K,\epsilon)]}\big(|Y_{0}-Y'_{0}|+|Z_{0}-Z'_{0}|\big),$$
for every $Y_{0},Y'_{0},Z_{0},Z'_{0} \in \mathbb{R}$ and $\omega \in \Omega$.
%Then by previous lemma, $$\mathcal{M}^{\epsilon}(\omega)=\big\{\big(l^{\epsilon}(\omega,Y_{0},Z_{0}),Y_{0},Z_{0}\big)^{T}:Y_{0},Z_{0}\in\mathbb{R}\big\}.$$
Utilizing Theorem III.9 in Casting and Valadier \cite[p.67]{castaing1977convex}, $\mathcal{M}^{\epsilon}(\omega)$ is a random set, i.e., for any $U=(X,Y,Z)^{T}\in \mathbb{R}^{3}$,
\begin{equation}\label{eq10}
\omega\mapsto \mathop {\mbox{inf} }\limits_{U'\in \mathbb{R}^3}\big|U-\big(l^{\epsilon}(\omega,\mathfrak{J}U'),\mathfrak{J}U'\big)^{T}\big|,
\end{equation}
is measurable. The space $\mathbb{R}^3$ has a countable dense subset $\mathbb{Q}^3$. Then the infimum in \eqref{eq10} is equivalent to
\begin{equation*}
\inf\limits_{U'\in \mathbb{Q}^3}\big|U-\big(l^{\epsilon}(\omega,\mathfrak{J}U'),\mathfrak{J}U'\big)^{T}\big|.
\end{equation*}
Under the infimum in \eqref{eq10} the measurability of any expression can be determined, since for all $U'$ in $\mathbb{R}^3$ the map $\omega\mapsto l^{\epsilon}(\omega,\mathfrak{J}U')$ is measurable. The slow flow is found as a random graph of $l^{\epsilon}$.\\
\indent Finally we need to show that $\mathcal{M}^{\epsilon}(\omega)$ is positively invariant, i.e., for all $U_{0} =(X_{0},Y_{0},Z_{0})^{T}$ in $\mathcal{M}^{\epsilon}(\omega),$ $ U(s,\omega,U_{0})$ is in $\mathcal{M}^{\epsilon}(\theta_{s}\omega)$ for every $s\geq 0.$ Note that $U(t+s,\omega,U_{0})$ is a solution of \begin{align*}
&dX=\frac{1}{\epsilon}g_{1}(X+\sigma\eta^{\epsilon}(\theta_{t}\omega), Y, Z)dt,\\&dY=g_{2}(X+\sigma \eta^{\epsilon}(\theta_{t}\omega), Y, Z)dt,\\&dZ=g_{3}(X+\sigma \eta^{\epsilon}(\theta_{t}\omega), Y, Z)dt,\end{align*} with initial value $U(s)=\big(X(s),Y(s),Z(s)\big)^{T}=U(s,\omega,U_{0})$. So, $U(t+s,\omega,U_{0})=U\big(t,\theta_{s}\omega,U(s,\omega,U_{0})\big)$. Since $U(\cdot,\omega,U_{0})$ is in $C_{-\frac{\gamma}{\epsilon}}^{\mathbb{R}^{3},-}$, we gain
$U\big(\cdot,\theta_{s}\omega,U(s,\omega,U_{0})\big)\in C_{-\frac{\gamma}{\epsilon}}^{\mathbb{R}^{3},-}$. Hence, $U(s,\omega,U_{0}) \in \mathcal{M}^{\epsilon}(\theta_{s}\omega)$.
\end{proof}

\section{Example}\label{E}
Namely, we consider the stochastic system
\begin{equation}\label{ESK}
\left\{
  \begin{array}{ll}
    \dot{x}=\frac{1}{\epsilon}(-x^{3}+3x-10y+5z+3)+\frac{\sigma}{\sqrt[\alpha]{\epsilon}}\dot{L}_{t}^{\alpha}, & \hbox{\text{in}\,\,$\mathbb{R}$,} \\
    \dot{y}=x-2y+z, & \hbox{\text{in}\,\,$\mathbb{R}$,} \\
    \dot{z}=y-z, & \hbox{\text{in}\,\,$\mathbb{R}$,}
  \end{array}
\right.
\end{equation}
where $x$ is the ``fast" component, $(y,z)$ is the ``slow" component, $k=-10$ and $\lambda(z)=-5z-3$.
\begin{figure}[H]
\begin{center}
  \begin{minipage}{2.1in}
\leftline{(a)}
\includegraphics[width=2.1in]{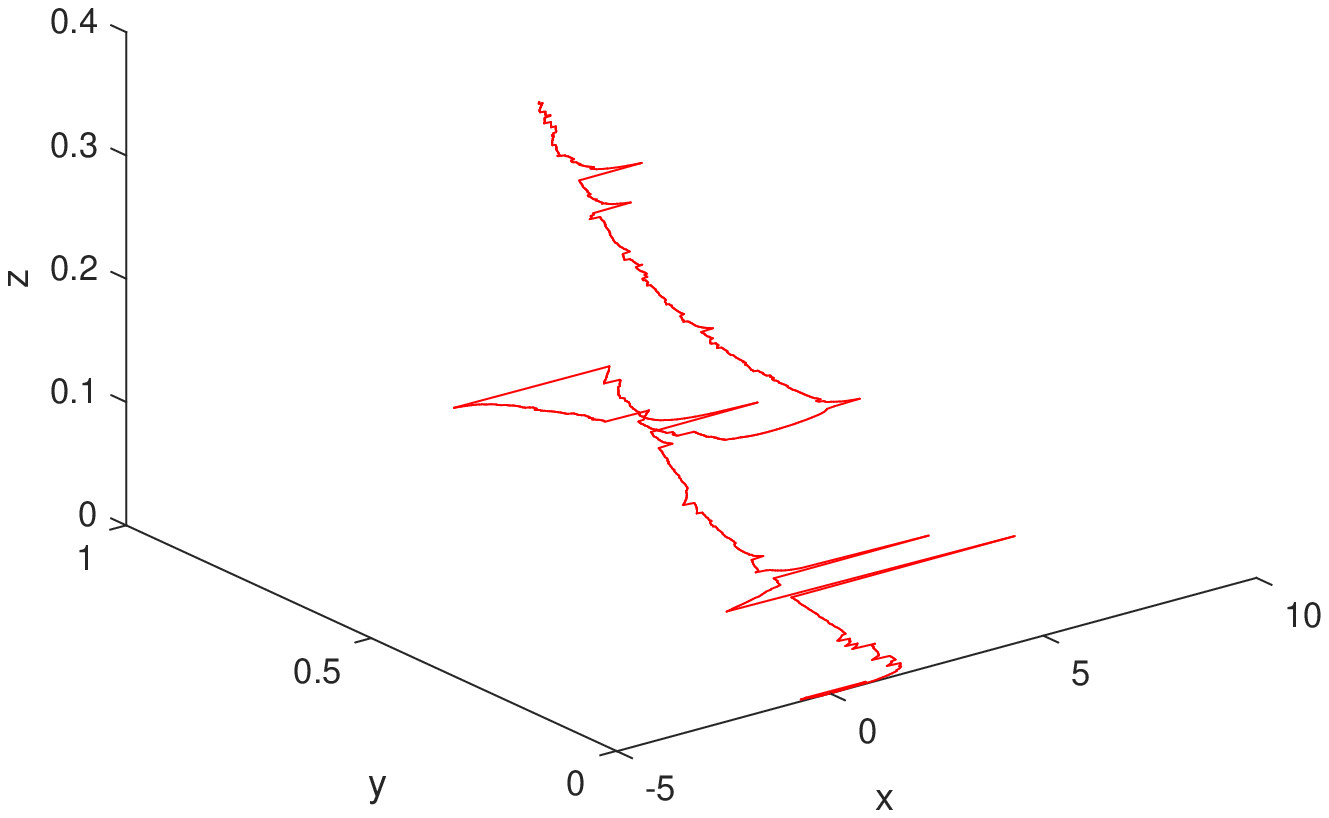}
\end{minipage}
\hfill
\begin{minipage}{2.1in}
\leftline{(b)}
\includegraphics[width=2.1in]{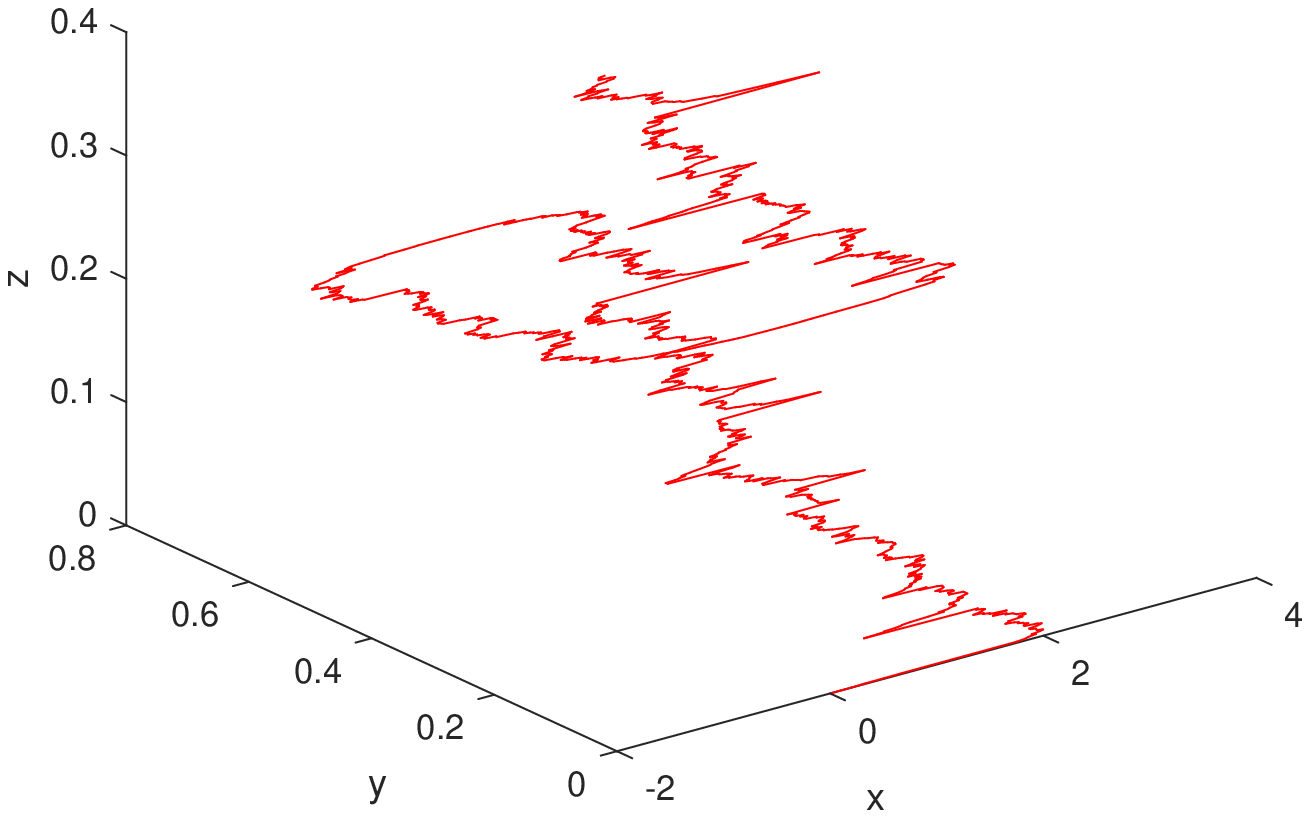}
\end{minipage}
\hfill
  \begin{minipage}{2.1in}
\leftline{(c)}
\includegraphics[width=2.1in]{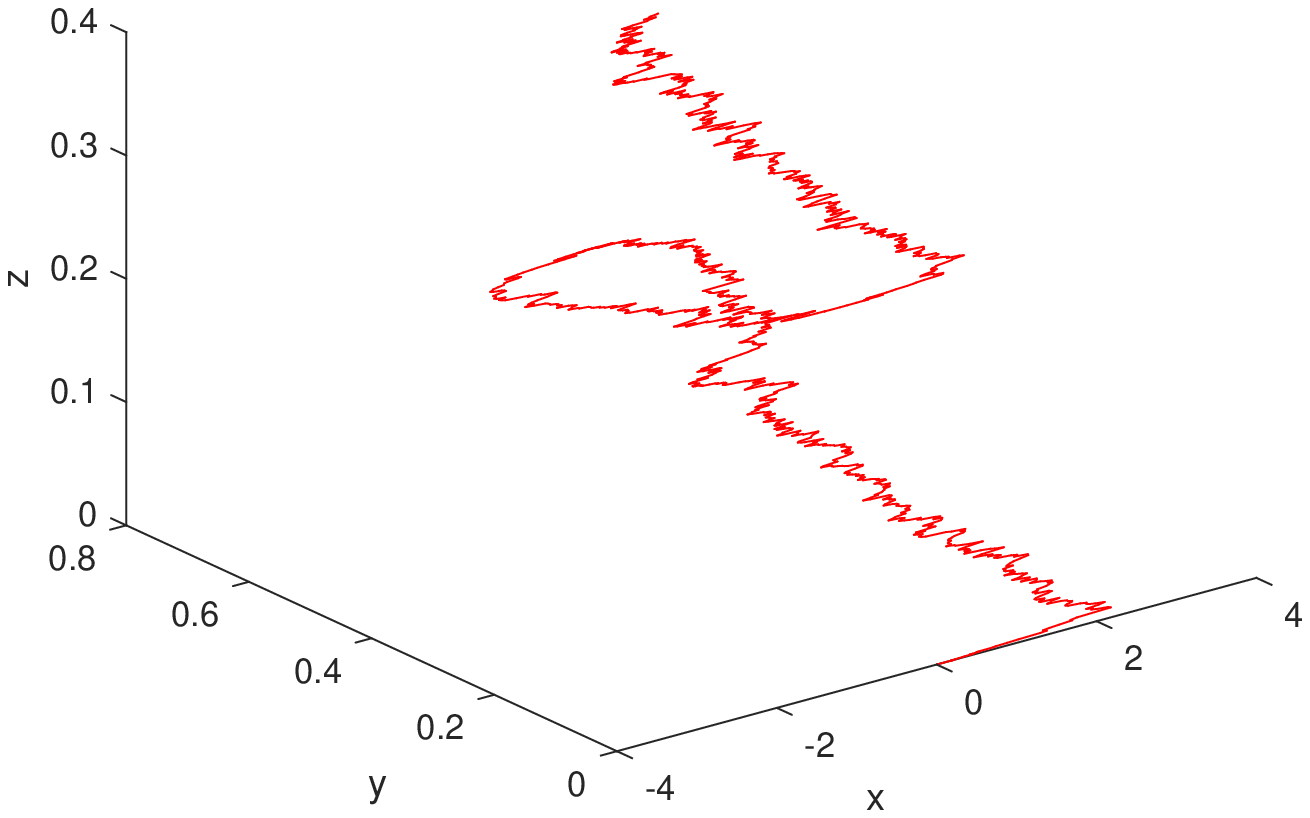}
\end{minipage}
\caption{The evolution of system \eqref{ESK} for the choice of
initial conditions $x(0)=y(0)=z(0)=0$ at the fixed noise intensity $\sigma=0.5$ with a scaling parameter $\epsilon=0.05$ and a gradual increase in the stability index: (a) $\alpha=0.8$; (b) $\alpha=1.6$; (c) $\alpha=1.9$.}\label{Fig1}
\end{center}
\end{figure}
If we fix $\sigma=0.5$ and increase $\alpha$ in system \eqref{ESK} for the choice of
initial conditions $x(0)=y(0)=z(0)=0$ with a scaling parameter $\epsilon=0.05$, we observe
the following typical sequence of events in Figure \ref{Fig1}. The stochastic nonlinear
dynamics with the stability index $\alpha=0.8$ display the complex spatio-temporal oscillations accompanied by few big jumps as depicted in Fig. \ref{Fig1}(a). With respect to $\alpha=1.6$, the external L\'evy noise leads to dramatically different dynamical behavior in the red trajectory, which is confirmed numerically in Fig. \ref{Fig1}(b).
 As the stability index $\alpha$ is increased
toward 1.9, the path of stochastic system \eqref{ESK} shows low frequency oscillations excited by L\'evy noise as illustrated in
Fig. \ref{Fig1}(c). For three different values of $\alpha$, stochastic noise of the the fast variable $x$  can significantly influence the
dynamics of the whole Koper model. However, the shapes of different trajectories from the viewpoint of the slow surface $(y,z)$ are consistent.

\begin{figure}[H]
\begin{center}
  \begin{minipage}{2.1in}
\leftline{(a)}
\includegraphics[width=2.1in]{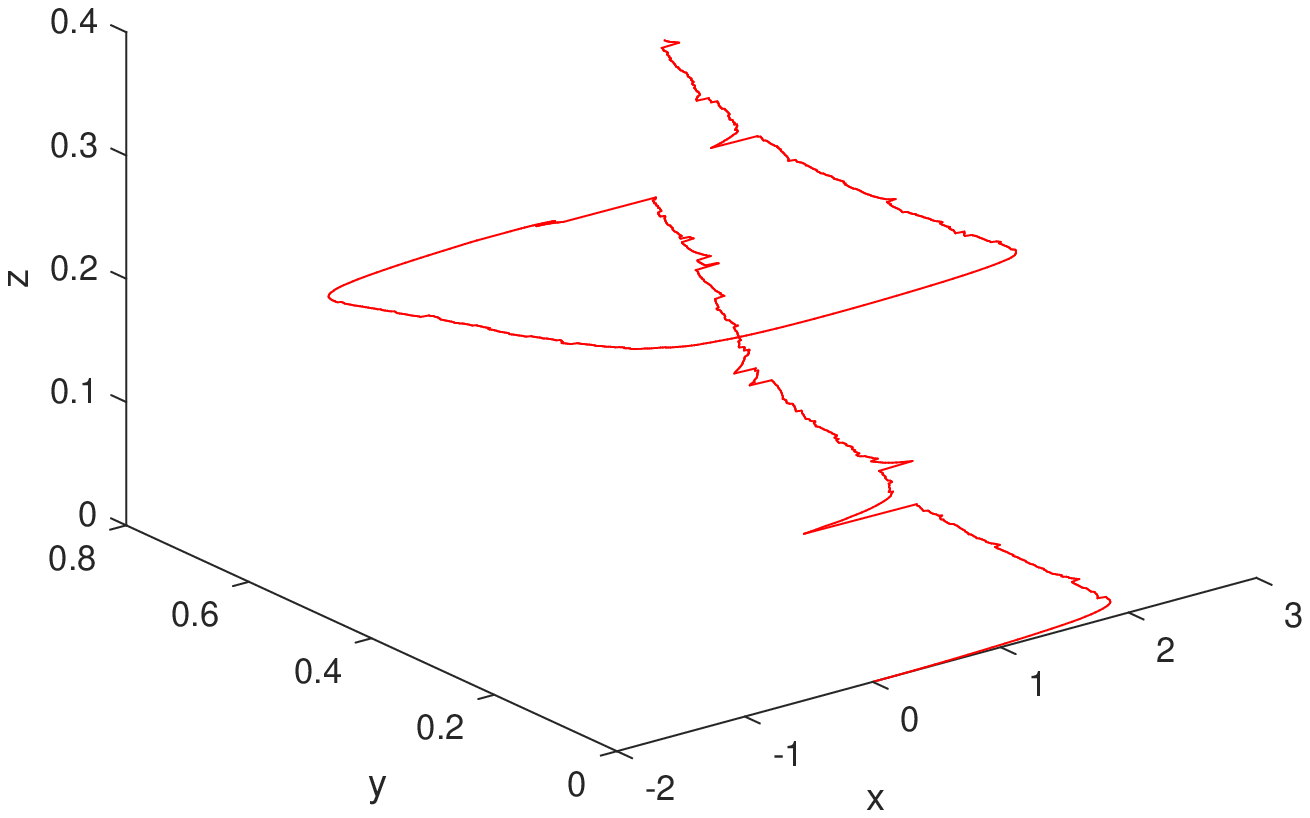}
\end{minipage}
\hfill
\begin{minipage}{2.1in}
\leftline{(b)}
\includegraphics[width=2.1in]{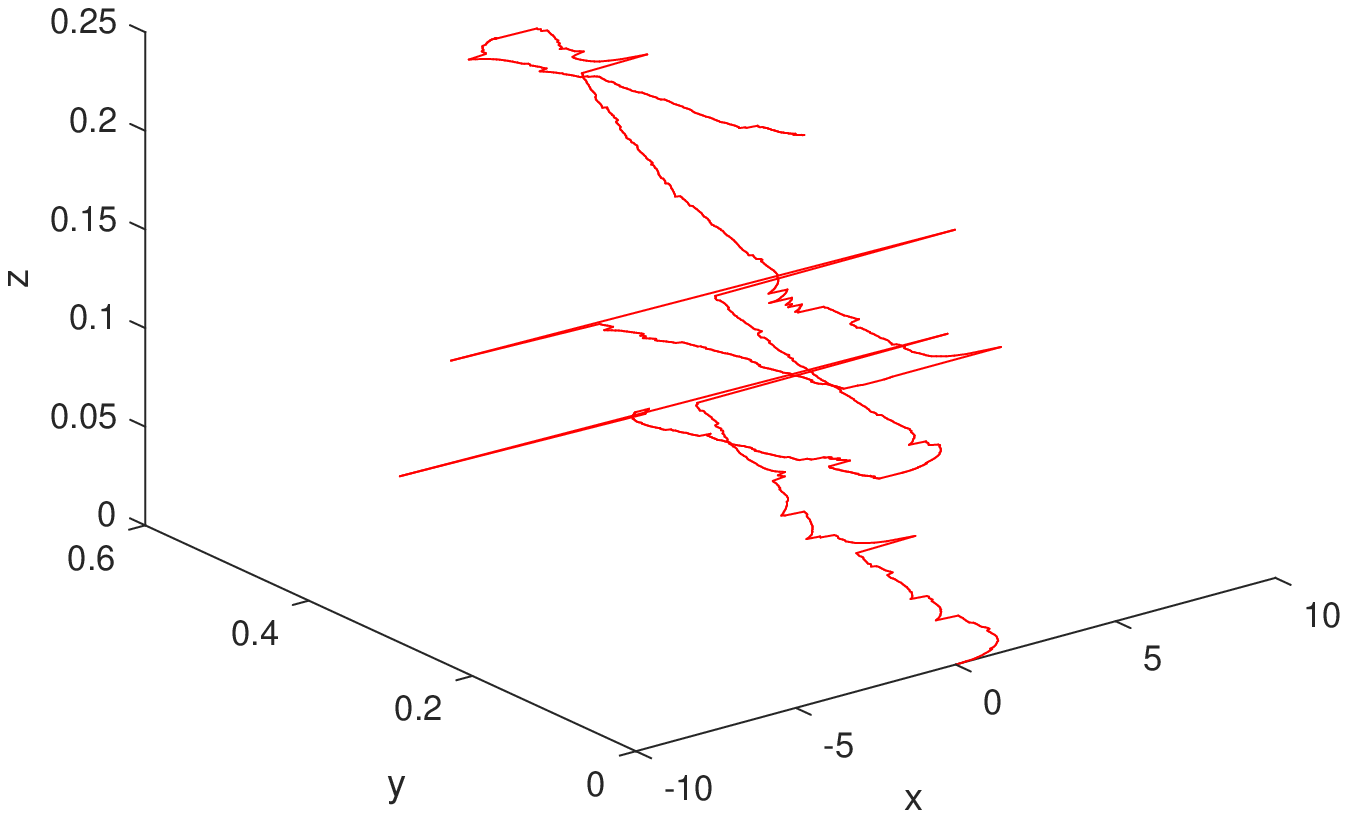}
\end{minipage}
\hfill
  \begin{minipage}{2.1in}
\leftline{(c)}
\includegraphics[width=2.1in]{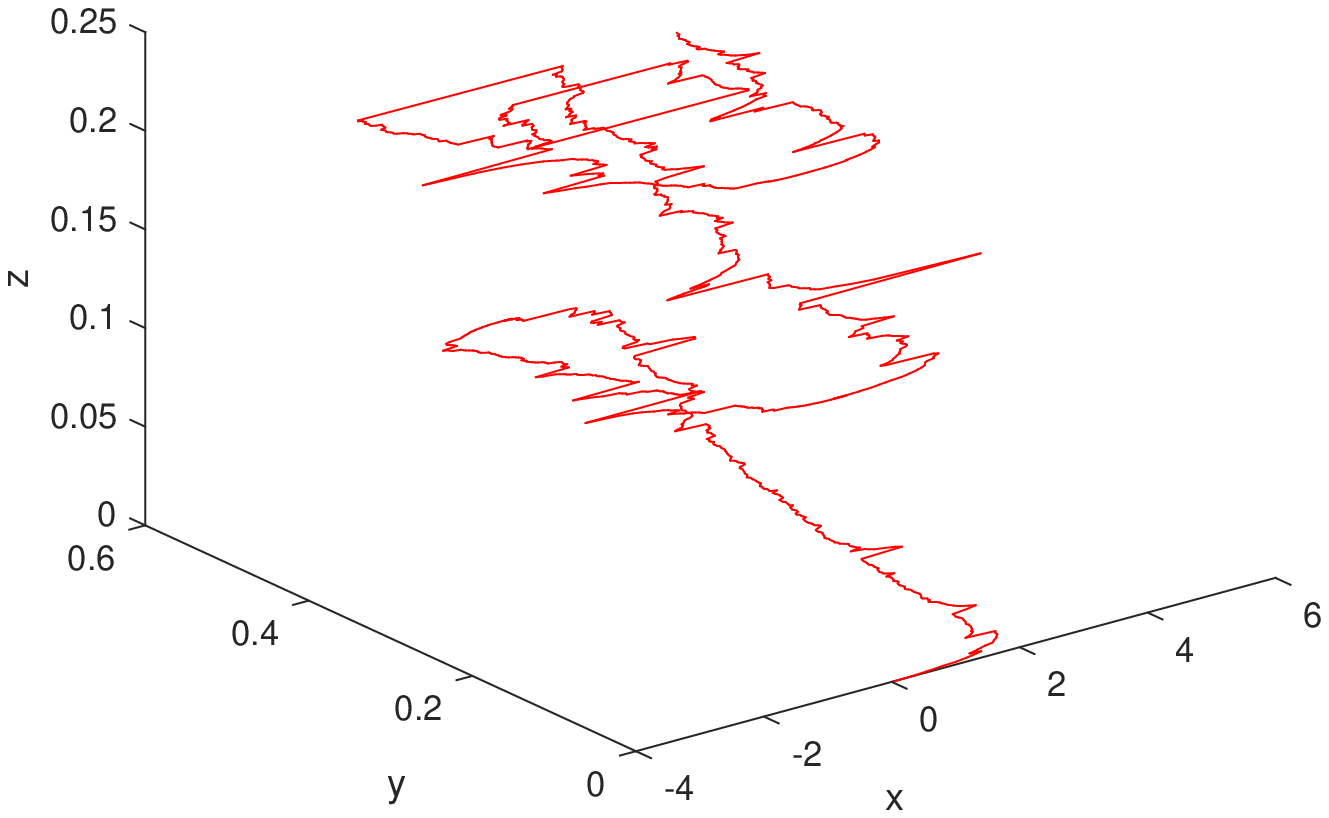}
\end{minipage}
\caption{The evolution of system \eqref{ESK} starting from $x(0)=y(0)=z(0)=0$ with a fixed stability index
$\alpha=1$ and a scaling parameter $\epsilon=0.05$ as the noise intensity increases: (a) $\sigma=0.1$; (b) $\sigma=0.5$; (c) $\sigma=0.8$.}\label{Fig2}
\end{center}
\end{figure}
Now we fix $\alpha=1$ and consider a variation of $\sigma$ in system \eqref{ESK} starting from $x(0)=y(0)=z(0)=0$ with a scaling parameter $\epsilon=0.05$; see Fig. \ref{Fig2}. For small enough $\sigma=0.1$, the path has spiral pattern with small-amplitude fluctuations, which  is clearly demonstrated in Fig. \ref{Fig2}(a). When $\sigma$-value just arrives at 0.5, the status of the path changes happens abruptly for stochastic system \eqref{ESK} as plotted in Fig. \ref{Fig2}(b). As the noise intensity $\sigma$ increases further to 0.8, large variations in the dynamics of the trajectory are indicated in Fig. \ref{Fig2}(c) because of strong external noise. Based on the effects of increasing noise intensities, the curves become more and more sophisticated.

 If we scale the time $t\rightarrow\epsilon t$
and use the self-similarity $\epsilon^{-1/\alpha}L_{\epsilon t}^{\alpha}\overset{d}{=}L_{t}^{\alpha}$, then stochastic system \eqref{ESK} in the sense of distribution is equal to
\begin{equation*}
\left\{
  \begin{array}{ll}
    dx=(-x^{3}+3x-10y+5z+3)dt+\sigma dL_{t}^{\alpha}, & \hbox{\text{in}\,\,$\mathbb{R}$,} \\
  dy=\epsilon(x-2y+z)dt, & \hbox{\text{in}\,\,$\mathbb{R}$,} \\
   dz=\epsilon(y-z)dt, & \hbox{\text{in}\,\,$\mathbb{R}$.}
  \end{array}
\right.
\end{equation*}
When $\sigma=0$, the deterministic system
\begin{equation}\label{Dks}
\left\{
  \begin{array}{ll}
     \dot{x}=(-x^{3}+3x-10y+5z+3), & \hbox{\text{in}\,\,$\mathbb{R}$,} \\
  \dot{y}=\epsilon(x-2y+z), & \hbox{\text{in}\,\,$\mathbb{R}$,} \\
    \dot{z}=\epsilon(y-z), & \hbox{\text{in}\,\,$\mathbb{R}$.}
  \end{array}
\right.
\end{equation}
has a unique fixed point $P=(1,1,1)$.  Linearize by finding
the Jacobian matrix. Hence
\begin{equation*}
J=\left(
    \begin{array}{ccc}
      3(1-x^2) & -10 & 5 \\
      \epsilon & -2\epsilon & \epsilon \\
      0 & \epsilon & -\epsilon \\
    \end{array}
  \right).
\end{equation*}
Therefore,
\begin{equation*}
J_P=\left(
    \begin{array}{ccc}
      0 & -10 & 5 \\
      \epsilon & -2\epsilon & \epsilon \\
      0 & \epsilon & -\epsilon \\
    \end{array}
  \right).
\end{equation*}
The stability of the equilibrium point $P$ is determined by the associated Jacobian matrix $J_P$. The trace $\text{tr}(J_P)=-3\epsilon$ is negative, and the determinant $\det(J_P)=-5\epsilon^2$ is also negative. It follows that three eigenvalues are negative. Thus, $P$ is a stable fixed point.
\begin{figure}[!htp]
	\begin{minipage}{0.48\linewidth}
		\leftline{(a)}
		\centerline{\includegraphics[height = 7.5cm, width = 9.5cm]{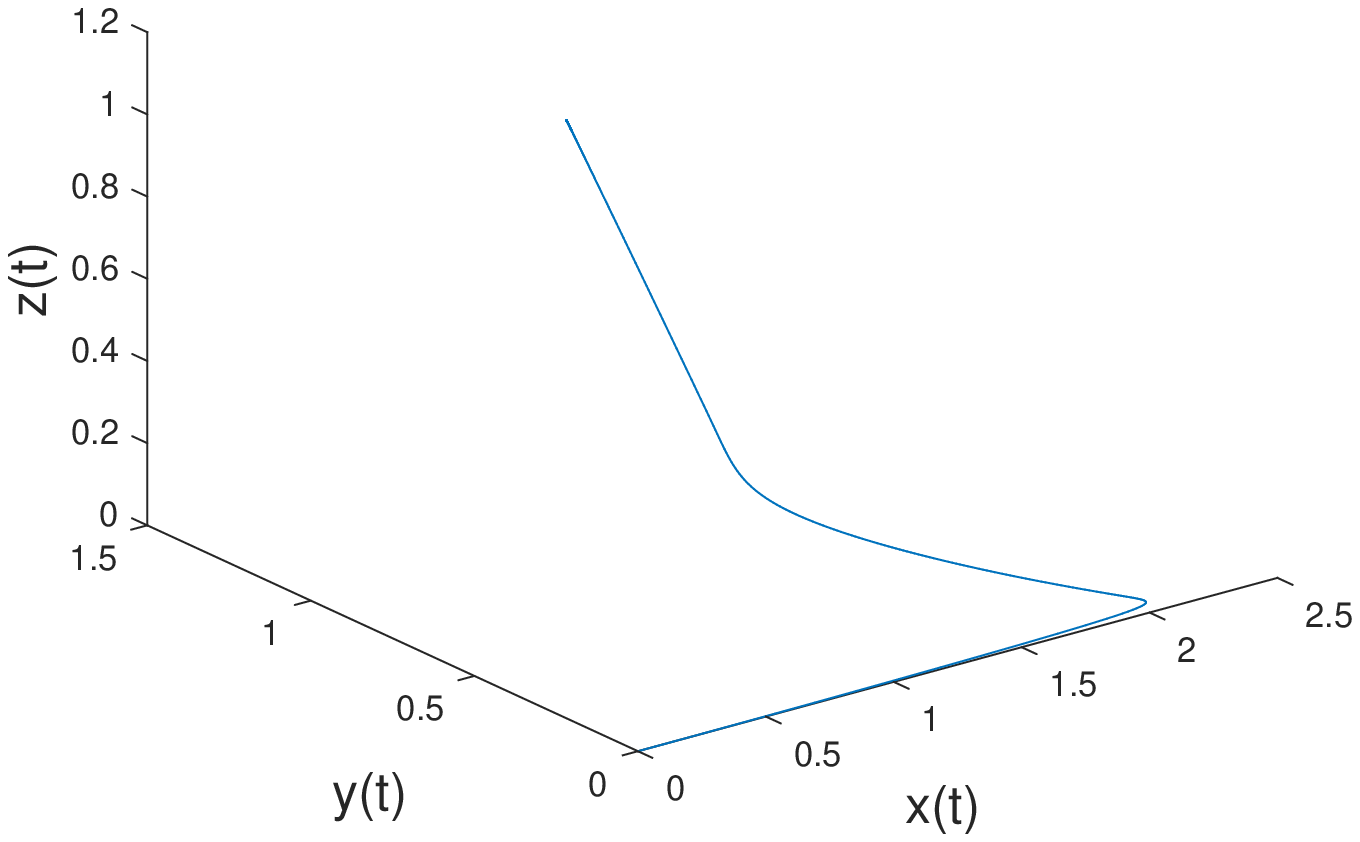}}
	\end{minipage}
	\hfill
	\begin{minipage}{0.48\linewidth}
		\leftline{(b)}
		\centerline{\includegraphics[height = 7.5cm, width = 9.5cm]{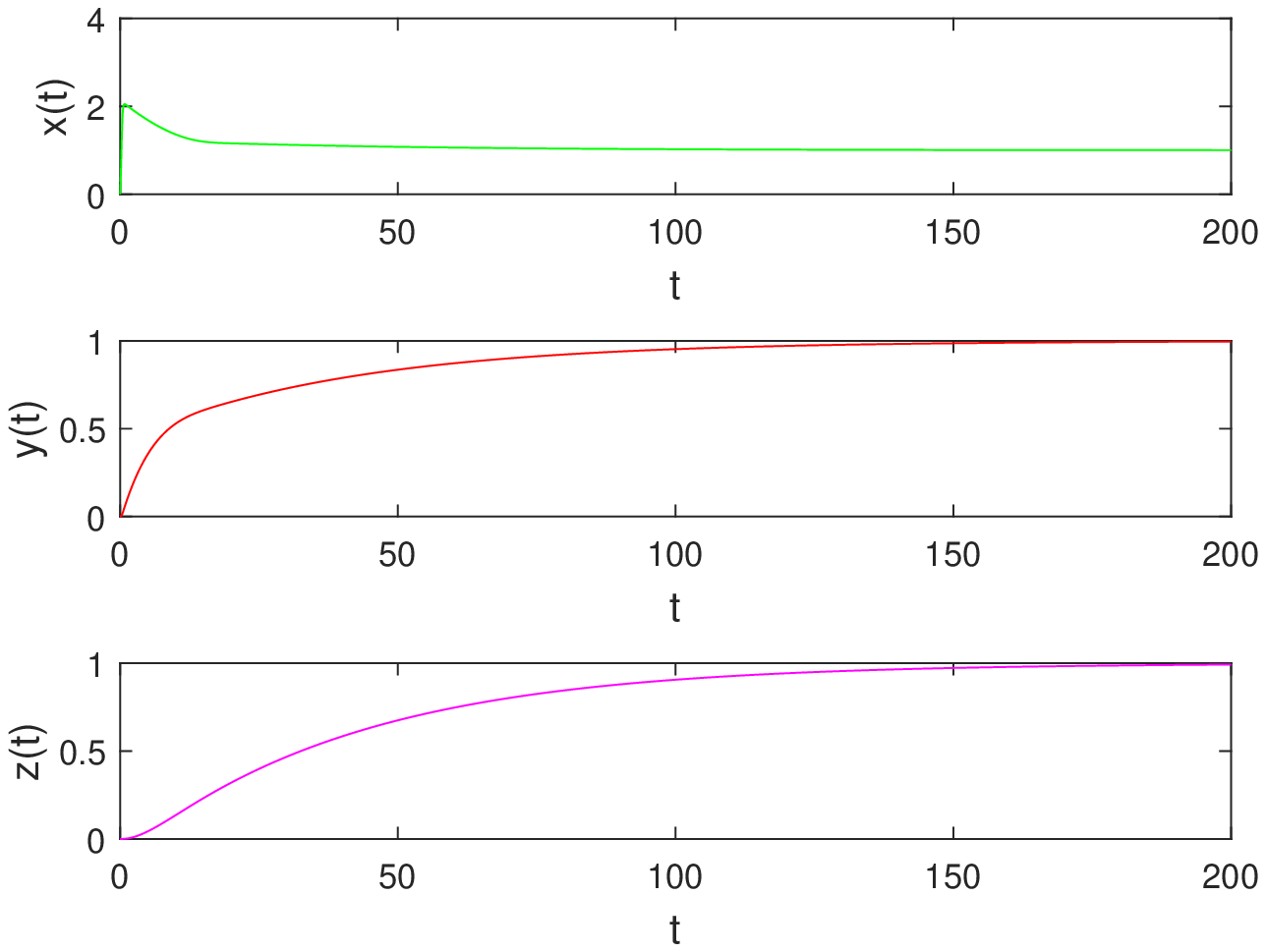}}
	\end{minipage}
	\caption{ When $x(0)=y(0)=z(0)=0$ and $\epsilon=0.05$: (a) one trajectory of system \eqref{Dks};
		(b) dynamical behaviors for the three variables $x(t)$, $y(t)$ and $z(t)$ in system \eqref{Dks}.
	}\label{Fig3}
\end{figure}

Observe that in the beginning the blue curve of system \eqref{Dks} with $\epsilon=0.05$ departs from $x(0)=y(0)=z(0)=0$ along the direction of $x$-axis, shown in Figure \ref{Fig3}(a).
But some time later, there is a gentle growing tendency for
the trajectory at $x(t)=2$ to move upward. With the increase of time, the path encounters a crucial turning point and has a precipitous climb to the stable fixed point $P$. More significantly, it is parallel to the
$(y,z)$-plane of the slow variables in this situation. By comparison with Fig. \ref{Fig1}-\ref{Fig2}, the Koper system \eqref{ESK} exhibits sensitivity to stochastic disturbance.

We sketch three time series data for the system \eqref{Dks} when $x(0)=y(0)=z(0)=0$ and $\epsilon=0.05$, as clearly detailed in Figure \ref{Fig3}(b).
It is manifest that the fast variable $x(t)$ jumps to 2 instantaneously at first, and then decreases to 1 rapidly and suddenly.
The green curve about $x(t)$ stays at the same level after $t=15$. From the time series drawings of the slow variables $y(t)$ and $z(t)$, they
both grow to 1 eventually. The red path about $y(t)$ increases much quicker than the purple trajectory about $z(t)$.
\section{Conclusions and future challenges}\label{Cf}

We investigated three-dimensional stochastic slow-fast Koper system \eqref{K} driven by $\alpha$-stable L\'evy noise.
We constructed the slow manifold in which the fast variable $x$ can be expressed as the random function of the two slow variables $y$ and $z$.
Stochastic nonlinear dimensionality reduction helped us make more accurate predictions by using stochastic differential equations
for the slow variables. We carried out the computation of the practical Koper models \eqref{ESK} and \eqref{Dks}.

When the scaling parameter $\hat{\epsilon}$ is sufficiently small, i.e., $0<\hat{\epsilon}\ll1$, stochastic Koper model \eqref{K} has three times scales: variable $x$ is fast, $y$ is slow, and $z$ is slower. It indicates the ratio of three times scales such that $|\frac{dx}{dt}|\gg|\frac{dy}{dt}|\gg|\frac{dz}{dt}|$. So it is meaningful to project high-dimensional dynamics onto lower-dimensional effective manifolds in this case.

What happens if the main bifurcation parameters $k$ and $\lambda(z)=:\lambda$ change? The folded node/focus and supercritical Hopf bifurcation are expected to occur in parameter space for the Koper model without noise.
By classifying the type and stability of the equilibrium points, we
obtain a bifurcation diagram under parameter variation.

We can generalize the consideration to a number of cases where the variables $x$, $y$ and $z$ both are pertubated by $\alpha$-stable L\'evy noises,  even to the extent that the influences are the more general L\'evy processes including multiplicative and additive effects. We may use It\^{o}, Stratonovich or Marcus type stochastic differential equations.

The orbits can escape from the region of the metastable equilibrium if we face random perturbation \cite{LYX}. It is interesting to characterize the most probable transition from one metastable state to another \cite{HCYD,TYTB}. The random slow manifold still depends on $k$ and $\lambda$. It is a challenge to plot the bifurcation diagram in the $(k,\lambda)$-plane of the random slow flow. Luckily, we can explore stochastic bifurcations in dynamical systems driven by L\'evy noises with support for statistical modeling and computation \cite{TTYBD,YB,YZD}.

We often have difficulty comprehending data in phase space with the attracting or repelling random slow manifolds.
Thus it is useful for visualization purposes by reducing data to a small number of dimensions.

\bigskip

\noindent\textbf{DATA AVAILABILITY}

Numerical algorithms source code that support the findings of this study are openly
available in GitHub, Ref. \cite{SY}.

%\renewcommand{\bibname}{References} % change title name bibiliography to references
%\bibliographystyle{ieeetr}
%\fontsize{10}{10}\selectfont{\bibliography{references.bib}}
\end{document}